\documentclass[11pt]{amsart}
\usepackage{fullpage}
\usepackage{color}
\usepackage{pstricks,pst-node,pst-plot} 
\usepackage{graphicx,psfrag} 
\usepackage{color,hyperref}
\usepackage{tikz}
\usepackage{pgffor}
\usepackage{hyperref}

\usepackage{perpage}	 
	
\MakePerPage{footnote}	 

\newtheorem{theorem}{Theorem}[section]
\newtheorem{lemma}[theorem]{Lemma}
\newtheorem{corollary}[theorem]{Corollary}
\newtheorem{proposition}[theorem]{Proposition}

\newtheorem{sublemma}[theorem]{Sublemma}

\theoremstyle{definition}
\newtheorem{definition}[theorem]{Definition}
\newtheorem{example}[theorem]{Example}

\newtheorem{question}[theorem]{Question}

\newtheorem{open}[theorem]{Open Problem List}
\newtheorem{restrict}[theorem]{Restriction}

\theoremstyle{remark}
\newtheorem{remark}[theorem]{Remark}

\numberwithin{equation}{section}

\definecolor{gray}{rgb}{.5,.5,.5}

\definecolor{black}{rgb}{0,0,0}

\definecolor{blue}{rgb}{0,0,1}

\definecolor{red}{rgb}{1,0,0}
\def\red{\color{red}}
\definecolor{green}{rgb}{0,1,0}

\definecolor{yellow}{rgb}{1,1,.4}

\newrgbcolor{purple}{.5 0 .5}
\newrgbcolor{black}{0 0 0}
\newrgbcolor{white}{1 1 1}
\newrgbcolor{gold}{.5 .5 .2}
\newrgbcolor{darkgreen}{0 .5 0}
\newrgbcolor{gray}{.5 .5 .5}
\newrgbcolor{lightgray}{.75 .75 .75}

\begin{document}

\title{A correspondence between complexes and knots}

\author{Moshe Cohen}
\address{Department of Mathematics and Computer Science, Bar-Ilan University, Ramat Gan 52900, Israel}
\email{cohenm10@macs.biu.ac.il}

\date{August 2012}


\begin{abstract}
In \cite{Co:clock} the author investigates perfect matchings of a bipartite graph obtained from a knot diagram and demonstrates that these correspond to discrete Morse functions on a 2-complex for the 2-sphere.  This relationship is expounded below for the opposite audience: those who may be unfamiliar with knots.
\end{abstract}

\maketitle


The intent of this quick note is to make the study of knot theory more accesible to those studying discrete Morse theory and complexes.

\textbf{Acknowledgements.}  This correspondence was developed by the author during a summer school on ``Discrete Morse Theory and Commutative Algebra'' from July 18th through August 1st, 2012 at the Institut Mittag-Leffler (the Royal Swedish Academy of Sciences).  The author would like to thank the organizers Bruno Benedetti and Alexander Engstr\"{o}m as well as the staff of the facility.  \href{http://www.mittag-leffler.se/summer2012/summerschools/}{http://www.mittag-leffler.se/summer2012/summerschools/}

The author was also partially supported by the Oswald Veblen Fund and by the Minerva Foundation of Germany.

\section{Discrete Morse functions on a 2-complex of the 2-sphere}
\label{sec:DMT}

In the discussion below, let $\Delta$ be any 2-complex of the 2-sphere; see Open Problem List \ref{note:open} (1) and (2) for open problems that generalize this to higher genus surfaces or higher dimensions.  Call its 0-, 1-, and 2-dimensional cells \emph{vertices}, \emph{edges}, and \emph{faces}, respectively, as usual.  This complex need not be simplicial; in particular multiple edges between the same pair of vertices, edges whose endpoints are both the same vertex, and faces of length less than three are all allowed.

Let $G$ be the graph given by the 1-skeleton of this complex; it can be realized by a plane embedding, and this is how the graph will be considered.

Let $\Delta^*$ be the \emph{dual block complex} with its 1-skeleton given as the graph $G^*$ realized by a plane embedding dual to the graph $G$.

Let $\mathcal{F}(\Delta)$ be the \emph{face poset} of $\Delta$.  Rather than view this as a usual poset with heights given by the dimension of the cells, take the embedding of the poset given by the plane embedding discussed above.  Call this plane tripartite graph $\widehat{\Gamma}$.

The 2-sphere has two \emph{critical cells}:  one $v_0$ of dimension 0 and one $f_0$ of dimension 2.  In the setting below these must satisfy the following additional requirement.

\begin{restrict}
\label{note:square}
The critical cells $v_0$ and $f_0$ must form a square face in $\widehat{\Gamma}$, the face poset $\mathcal{F}(\Delta)$.
\end{restrict}

See Open Problem List \ref{note:open} (3) for an open problem generalizing this for the setting without this condition.

Delete the critical cells $v_0$ and $f_0$ from $\widehat{\Gamma}$ (along with all incident edges) to obtain a new graph $\Gamma$.  This will be the graph on which perfect matchings will be considered.  In this setting the perfect matchings correspond to discrete Morse functions.

\section{Perfect matchings coming from knots}
\label{sec:knots}

A \emph{knot} $K$ is an embedding of $S^1$ into a three-manifold; here it will always be $S^3$, and the unfamiliar reader may choose to think of this as $\mathbb{R}^3$ with the one-point compactification at infinity sufficiently far from the embedding.  A \emph{link} is the embedding of several circles.

The \emph{projection graph} (or \emph{universe} $U$, following Kauffman's notation in \cite{Kauff}) is the projection of the knot onto the plane such that there are no tangencies and such that any intersection is a double point.  Thus $U$ is a 4-valent graph.

The graph can be transformed into a \emph{knot diagram} $D$ by including over- and under-crossing information at each of these vertices.  These are denoted by broken lines signifying the under-crossing.  The vertices of $U$ are now called \emph{crossings}.  According to a theorem by Reidemeister in 1926 (see for example \cite{Adams}), two knot diagrams represent the same knot if one can be transformed into the other by a sequence of three local changes called \emph{Reidemeister moves}.

Checkerboard color the faces of $U$ using the Jordan Curve Theorem.  Define the plane graph $G$ with vertices coming from the regions of $U$ that are colored black and with edges passing through the vertices of $U$ (i.e. the original crossings), connected using the planarity of the universe.  Define the plane dual $G^*$ similarly with the white regions of the checkerboard coloring.

This graph can be transformed into the often-studied \emph{signed Tait graph}, which by abuse of notation will also be called $G$.  The signs of the edges are determined according to Figure \ref{fig:SignedTaitGraphSignsTIKZ}.

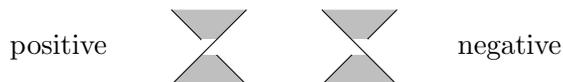
\begin{figure}[h]
\begin{center}

\begin{tikzpicture}
	\fill[gray!50!white] (0,0) -- +(.5,.5) -- +(1,0) -- cycle;
	\fill[gray!50!white] (0,1) -- +(.5,-.5) -- +(1,0) -- cycle;
	\draw[-] (1,0) -- +(-1,1);
	\fill[color=white] (.4,.4) rectangle +(.2,.2);
	\draw[-] (0,0) -- +(1,1);

	\fill[gray!50!white] (2,0) -- +(.5,.5) -- +(1,0) -- cycle;
	\fill[gray!50!white] (2,1) -- +(.5,-.5) -- +(1,0) -- cycle;
	\draw[-] (2,0) -- +(1,1);
	\fill[color=white] (2.4,.4) rectangle +(.2,.2);
	\draw[-] (3,0) -- +(-1,1);

	\draw (-1.5,.5) node {positive};
	\draw (4.5,.5) node {negative};
	
\end{tikzpicture}
	\caption{Crossings determine the sign of the edges in the signed Tait graph.}
	\label{fig:SignedTaitGraphSignsTIKZ}
\end{center}
\end{figure}

There is a bijection between the set of all knot diagrams and the set of all signed plane graphs $G$.  Spanning tree expansions of $G$ have been used to produce several models of use in knot theory (see references in \cite{Co:clock}).  
However not all the elements in these models correspond to combinatorial objects (like the spanning trees).  This leads one to ask for a better model.

The \emph{overlaid Tait graph} $\widehat{\Gamma}$ can be obtained by overlaying $G$ and $G^*$ so that dual edges intersect at the original crossings of the diagram $D$.  This graph is bipartite:  its black vertex set corresponds to the intersections of an edge of $G$ with its dual edge in $G^*$, and its white vertex set corresponds to the vertices of both the Tait graph $G$ and its plane dual $G^*$.  That is, $V(\widehat{\Gamma})=[E(G)\cap E(G^*)]\sqcup [V(G)\sqcup V(G^*)]$.  The edges of this graph are the half-edges of both Tait graphs.  A similar notion is found in work by Huggett, Moffatt, and Virdee \cite{HugVir}.

All of the black vertices of $\widehat{\Gamma}$ are four-valent, as these correspond with vertices of the universe $U$, and all of the faces of $\widehat{\Gamma}$ are square, as these correspond to edges of the universe $U$ with each edge between two vertices and two faces.  Thus this is a quadrangulation of the 2-sphere.  
%
%
%
%
%
Choose one such square face and star the two white vertices.  This condition is equivalent to Restriction \ref{note:square}.

In order to consider perfect matchings, delete the two starred white vertices to obtain the \emph{balanced overlaid Tait graph} $\Gamma$ that is the central graph for the present paper.  For more details of this construction see \cite{Co:jones} 
 or consider the following more succinct definition.

\begin{definition}
The \emph{balanced overlaid Tait graph} $\Gamma$ is a bipartite graph that can be obtained from a universe $U$ as follows.  Let every four-valent vertex in the universe $U$ be a black vertex in $\Gamma$.  Select two adjacent faces of $U$ and mark them $*$ by stars.  Let every non-starred face of $U$ be a white vertex in $\Gamma$.  A black vertex is adjcent to a white vertex whenever the vertex and face of $U$ are incident.
\end{definition}

Since $U$ is a plane graph, so is $\Gamma$.  Furthermore, all faces of $\Gamma$ are square except for the infinite face.  Let the boundary cycle of this infinite face be called the \emph{periphery}.  All black vertices not on the periphery are still four-valent.  
By considering the plane embedding of this graph, we may delete two white vertices and still recover all of $\widehat{\Gamma}$.

There is a another bijection between the set of (rooted) spanning trees (or arborescences) of $G$ and the set of perfect matchings (or dimer coverings) of $\Gamma$ that has been explored in previous work by the author \cite{Co:jones} and with Dasbach and Russell \cite{CoDaRu}, as well as in work by Kenyon, Propp, and Wilson \cite{KenPropp}.  
The graph $\Gamma$ is currently being studied by Kravchenko and Polyak \cite{KraPol} for knots on a torus in relation to cluster algebras.  
Dimers themselves have been studied extensively, as well;  see for example Kenyon's lecture notes \cite{Ken} on the subject.  

%

\section{Formalizing a correspondence between complexes and knots}
\label{sec:corresp}


\begin{proposition}
Discrete Morse functions on the 2-complex $\Delta$ of the 2-sphere correspond to perfect matchings on $\Gamma$ constructed from a knot diagram $D$.
\end{proposition}

\begin{proof}
This follows by construction with the 1-skeleton of $\Delta$ giving the Tait graph $G$ of the knot diagram and with the face poset $\mathcal{F}(\Delta)$ of the complex giving $\widehat{\Gamma}$.
\end{proof}

This leads to numerous questions translating properties back and forth between the two settings.  In particular, knots are studied and characterized by many invariants, some of which can be found on the Knot Atlas \cite{knotatlas} and on Knot Info \cite{knotinfo}.  See Open Problem List \ref{note:open} (4) and (5).

\begin{open}
\label{note:open}
Finally here are some open questions that are related to this correspondence:
\begin{enumerate}
	\item Extend to $\Delta$ for higher genus surfaces.  For knots on the torus see \cite{KraPol}.
	\item Extend to $\Delta$ in higher dimensions.
	\item Extend to $\Delta$ with critical cells that are not adjacent, that is, without imposing Restriction \ref{note:square}.  This should not be a problem in the discrete Morse setting, but it is problematic for certain topics of interest in the knot setting, specifically the Alexander polynomial (see \cite{CoDaRu}).
	\item Study knot invariants via properties of complexes.
	\item Study properties of complexes via knot invariants.
\end{enumerate}
\end{open}

\bibliographystyle{amsalpha}
\bibliography{12JulyBibliography-old}

\end{document}